\documentclass[11pt]{amsart}
\usepackage{comment}
\usepackage{amssymb,amsmath,amsthm,mathrsfs,amstext,amssymb}
\usepackage{dsfont}
\usepackage{tabularx}
\usepackage{array}
\usepackage[square,sort,comma,numbers]{natbib}
\usepackage{url} 
\usepackage{hyperref}
\usepackage{accents}
\usepackage[shortlabels]{enumitem}
\usepackage{tikz-cd}

\newtheorem{theorem}{Theorem}
\numberwithin{theorem}{section}
\newtheorem*{theorem*}{Theorem}

\newtheorem*{maintheorem*}{Main Theorem}
\newtheorem{lemma}[theorem]{Lemma}
\newtheorem*{lemma*}{Lemma}

\newtheorem*{fact*}{Fact}
\newtheorem{corollary}[theorem]{Corollary}
\newtheorem*{corollary*}{Corollary}
\newtheorem{proposition}[theorem]{Proposition}
\newtheorem*{proposition*}{Proposition}

\theoremstyle{definition}

\newtheorem{definition}[theorem]{Definition}
\newtheorem*{definition*}{Definition}
\newtheorem{remark}[theorem]{Remark}
\newtheorem*{remark*}{Remark}
\newtheorem{question}[theorem]{Question}
\newtheorem*{question*}{Question}

\newcommand{\A}{{\mathcal{A}}}
\newcommand{\B}{{\mathcal{B}}}

\newcommand{\F}{{\mathcal{F}}}
\newcommand{\G}{{\mathcal{G}}}

\newcommand{\I}{{\mathcal{I}}}

\newcommand{\M}{{\mathcal{M}}}

\newcommand{\EE}{{\mathbb{E}}}

\newcommand{\PP}{{\mathbb{P}}}

\renewcommand{\SS}{{\mathbb{S}}}

\renewcommand{\a}{{\mathfrak{a}}}

\renewcommand{\c}{{\mathfrak{c}}}

\DeclareMathOperator{\dom}{dom}

\DeclareMathOperator{\restr}{\upharpoonright}

\DeclareMathOperator{\partialto}{{\overset{\text{part}}{\to}}}

\newcommand{\simpleset}[1]{{\{{#1}\}}}

\newcommand{\set}[2]{{\{ {#1} \mid {#2} \}}}
\newcommand{\seq}[2]{{\langle {#1} \mid {#2} \rangle}}

\DeclareMathOperator{\extends}{{\mathbin{\leq}}}

\DeclareMathOperator{\forces}{{ \, \Vdash \, }}

\newcommand{\gen}{{\text{gen}}}

\newcommand{\finsubset}[1]{{[#1]^{<\omega}}}
\newcommand{\infsubset}[1]{{[#1]^{\omega}}}
\DeclareMathOperator{\Fin}{Fin}

\DeclareMathOperator{\non}{non}
\DeclareMathOperator{\cov}{cov}
\DeclareMathOperator{\cof}{cof}

\DeclareMathOperator{\pseudosubseteq}{\subseteq^*}

\DeclareMathOperator{\spec}{spec}
\DeclareMathOperator{\aE}{\a_{\text{\normalfont e}}}
\DeclareMathOperator{\av}{\a_{\text{\normalfont v}}}

\DeclareMathOperator{\aT}{\a_{\text{\normalfont T}}}
\newcommand{\functionspace}[2]{{^{#1}{#2}}}
\DeclareMathOperator{\bairespace}{{^\omega \omega}}

\DeclareMathOperator{\supp}{supp}

\DeclareMathOperator{\cofin}{cofin}

\DeclareMathOperator{\trace}{tr}

\title{Van Douwen and many non Van Douwen families}

\author{L. Schembecker}
\address{Department of Mathematics, University of Hamburg, Bundesstraße 55, 20146 Hamburg, Germany}
\email{lukas.schembecker@uni-hamburg.de}

\begin{document}	
	\maketitle
	
	\begin{abstract}
		We prove that the spectrum of Van Douwen families is closed under singular limits.
		For any maximal eventually different family Raghavan defined in \cite{Raghavan_2010} an associated ideal which measures how far the family is from being Van Douwen.
		Under {\sf CH} we prove that every ideal containing $\Fin$ is realized as the associated ideal of some maximal eventually different family.
		Finally, we construct maximal eventually different families with Sacks-indestructible associated ideals to prove that in the iterated Sacks-model every $\aleph_1$-generated ideal containing $\text{Fin}$ is realized.
	\end{abstract}

	\section{Introduction} \label{SEC_Introduction}

	Let $\spec(\aE)$ denote the set of all sizes of maximal eventually different families and $\aE$ its minimum (see Definition~\ref{DEF_MED}).
	Of particular interest for us are the following two well-known open questions.
	Since $\non(\M) \leq \aE$ \cite{BartoszynskiJudah_1995} and $\a < \non(\M)$ is consistent (for example in the random model), also $\a < \aE$ is consistent.
	However, the consistency of the other direction is an open problem:
	
	\begin{question} \label{QUE_A_Leq_Ae}
		Does ${\sf ZFC}$ prove $\a \leq \aE$?
	\end{question}
	
	Secondly, among various types of combinatorial families it seems to be a recurrent feature that their spectra are closed under singular limits.
	Hechler proved this property for mad families in~\cite{Hechler_1972} and Brian recently verified this also for partitions of Baire space into compact sets in \cite{Brian_2021}.
	The analogous question for many other types of combinatorial families is still open, for example for independent families, cofinitary groups or maximal eventually different families:
	
	\begin{question} \label{QUE_Closure}
		Is $\spec(\aE)$ closed under singular limits?
	\end{question}
	
	Instead of answering these questions for maximal eventually different families, we will instead consider them for the strengthened notion of Van Douwen families.
	A maximal eventually different family is called Van Douwen iff it is also maximal with respect to infinite partial functions (see Definition~\ref{DEF_VanDouwen}).
	Van Douwen asked whether families with this strong kind of maximality always exists (see problem 4.2 in Miller's problem list \cite{Miller_1993}).
	Zhang proved in \cite{Zhang_1999} that Van Douwen families of desired sizes may be forced by a c.c.c.\ forcing, so that {\sf MA} implies the existence of a Van Douwen family of size $\c$.
	Later, Raghavan \cite{Raghavan_2010} proved that indeed there always is a Van Douwen family of size $\c$.
	So, let $\spec(\av)$ be the set of sizes of Van Douwen families and $\av$ its minimum.
	It is not hard to prove and well-known that $\a \leq \av$ holds (see Corollary~\ref{COR_A_LEQ_AV}), so Question~\ref{QUE_A_Leq_Ae} has a positive answer for Van Douwen families.
	Using a similar argument In Lemma~\ref{LEM_MED_Forcing_Adds_Mad_Family} we show that the standard forcing $\EE_\F(I)$ for extending an eventually different family $\F$ by $I$-many elements to a maximal eventually different family of size $\max(\left|\F\right|, \left|I\right|)$ also adds a maximal almost disjoint family of the same size:
	
	\begin{lemma*}
		Let $\F$ be an eventually different family and $I$ an uncountable index set.
		Then
		$$
		\EE_\F(I) \forces \max(\left|\F\right|, \left|I\right|) \in \spec(\a).
		$$
	\end{lemma*}
	
	Hence, the standard forcing for realizing a desired spectrum of $\aE$ also forces $\a$ to have the same spectrum and thus cannot be used to separate them.
	Further, as it is the case for $\a$ and $\aT$, in Theorem~\ref{THM_Closure} we show that Question~\ref{QUE_Closure} also has a positive answer for Van Douwen families:
	
	\begin{theorem*}
		$\spec(\av)$ is closed under singular limits.
	\end{theorem*}

	Clearly, we have that $\spec(\av) \subseteq \spec(\aE)$.
	One of the central open questions regarding Van Douwen families is if we always have equality:
	
	\begin{question}
		Does $\spec(\av) = \spec(\aE)$ hold?
	\end{question}

	Notice, that a positive answer together with our Theorem~\ref{THM_Closure} would yield a positive answer for the well-known open Question~\ref{QUE_Closure}.
	Moreover, in order to answer Question~\ref{QUE_A_Leq_Ae} the following weaker version would suffice:

	\begin{question}
		Does $\av = \aE$ hold?
	\end{question}

	Towards an answer to this question, it is interesting to also study non Van Douwen families as well as their indestructibility via forcing.
	For any maximal eventually different family $\F$ in \cite{Raghavan_2010} Raghavan defined an associated ideal $\I_0(\F)$ which measures how far the family is from being Van Douwen (see Definition~\ref{DEF_Associated_Ideal}).
	This ideal is always proper and contains $\text{Fin}$ (see Proposition~\ref{PROP_Associated_Ideal}) and all ideals in this paper are assumed to be of this form.
	We prove that under {\sf CH} any ideal may be realized as the associated ideal of some maximal eventually different family (see Theorem~\ref{THM_Realize_Ideal}), i.e.\ there are many different maximal eventually different but non Van Douwen families:
	
	\begin{theorem*}
		Assume {\sf CH} and let $\I$ be an ideal.
		Then there is a maximal eventually different family such that $\I = \I_0(\F)$.
	\end{theorem*}

	Finally, towards the potential consistency of $\aE < \av$ we also show that the associated ideal $\I_0(\F)$ of a maximal eventually different family $\F$ may exhibit some forcing indestructibility, i.e.\ in the forcing extension $V[G]$ we have that $(\I_0(\F))^{V[G]}$ is the ideal generated by $(\I_0(\F))^V$. Note, that this also implies that the maximality of $\F$ is preserved.
	In particular, under {\sf CH} we show that the associated ideals may also be Sacks-indestructible (see Theorem~\ref{THM_Realize_Ideal_Sacks}):
	
	\begin{theorem*}
		Assume {\sf CH} and let $\I$ be an ideal.
		Then there is a maximal eventually different family such that $\I = \I_0(\F)$ and $\I_0(\F)$ is indestructible by any countable support iteration or product of Sacks-forcing.
	\end{theorem*}

	As an immediate corollary (see Corollary~\ref{COR_Realize_Ideal_Sacks_Model}) all $\aleph_1$-generated ideals are realized in the iterated Sacks-model:

	\begin{corollary*}
		In the iterated Sacks-model for every $\aleph_1$-generated ideal $\I$ there is a maximal eventually different family $\F$ such that $\I = \I_0(\F)$.
	\end{corollary*}

	\section{Preliminaries} \label{SEC_Preliminaries}

	\begin{definition}\label{DEF_MED}
		We say $f, g \in \bairespace$ are eventually different iff $\set{n < \omega}{f(n) = g(n)}$ is finite.
		A family $\F \subseteq \bairespace$ is called eventually different (e.d.) iff all $f \neq g \in \F$ are eventually different.
		It is called maximal (m.e.d.) iff it is maximal with respect to inclusion.
		Finally, we define the associated spectrum and cardinal characteristic
		\begin{align*}
			\spec(\aE) &:= \set{\left|\F\right|}{\F \text{ is a m.e.d.\ family}},\\
			\aE &:= \min(\spec(\aE)).
		\end{align*}
	\end{definition}
	
	\begin{remark}\label{REM_MED_Other_Domains}
		For any countably infinite $A, B$ we may equivalently consider m.e.d.\ families $\F \subseteq \functionspace{A}{B}$ by using bijections with $\omega$.
		In most cases $A, B = \omega$, however we will also consider the cases $A \in \infsubset{\omega}$ and $B = \omega \times \omega$.
	\end{remark}
	
	\begin{definition}\label{DEF_VanDouwen}
		Let $\F \subseteq \bairespace$ and $A \in \infsubset{\omega}$.
		Then we define $\F \restr A := \set{f \restr A}{f \in \F}$.
		We call $\F$ Van Douwen iff $\F \restr A$ is a m.e.d.\ family for all $A \in \infsubset{\omega}$.
		Analogously, we define the associated spectrum and cardinal characteristic
		\begin{align*}
			\spec(\av) &:= \set{\left|\F\right|}{\F \text{ is Van Douwen}},\\
			\av &:= \min(\spec(\av)).
		\end{align*}
	\end{definition}

	Clearly, we have $\spec(\av) \subseteq \spec(\aE)$, so also $\aE \leq \av$.
	Unlike as for other notions of strong maximality, such as $\omega$-maximality or tightness, Raghavan \cite{Raghavan_2010} proved that there always exists a Van Douwen family of size $\c$, i.e.\ the cardinal characteristic $\av$ is well-defined.
	Next, we present a short argument of the well-known fact that $\a \leq \av$ holds.
	
	\begin{definition} \label{DEF_Covered_reals}
		Let $\F$ be an e.d.\ family.
		Then we define
		$$
			\cov(\F) := \set{g \in \bairespace}{\exists \F_0 \in \finsubset{\F}\ \exists N < \omega \ \forall n \geq N \ \exists f \in \F_0 \ f(n) = g(n)},
		$$
		i.e.\ $g \in \cov(\F)$ iff its graph can almost be covered by finitely many elements of $\F$.
		Further, we set $\cov^+(\F) := \bairespace \setminus \cov(\F)$.
	\end{definition}
	
	\begin{proposition} \label{PROP_Uncovered_real}
		Let $\F$ be an e.d.\ family, $g \in \bairespace$, $\F_1 \in \infsubset{\F}$ and assume $g =^\infty f$ for every $f \in \F_1$.
		Then $g \in \cov^+(\F)$.
	\end{proposition}

	\begin{proof}
		Let $\F_0 \in \finsubset{\F}$ and $N_0 < \omega$.
		Choose $f \in \F_1 \setminus \F_0$.
		Since $\F$ is e.d.\ choose $N_1 \geq N_0$ such that $f(n) \neq f_0(n)$ for all $f_0 \in \F_0$ and $n \geq N_1$.
		Since $g =^\infty f$ choose $n \geq N_1$ such that $f(n) = g(n)$.
		Hence, $g(n) \neq f_0(n)$ for all $f_0 \in \F_0$.
		Thus, $g \in \cov^+(\F)$.
	\end{proof}
	
	\begin{proposition}\label{PROP_VanDouwen_Implies_Mad}
		Let $\F$ be a Van Douwen family and $g \in \bairespace$.
		For every $f \in \F$ define
		$$
			E_g^f := \set{n < \omega}{f(n) = g(n)}.
		$$
		Then, $\A^\F_g := \set{E_g^f}{f \in \F \text{ such that } E_g^f \in \infsubset{\omega}}$ is a (possibly finite) mad family.
		Further, $\A_g^\F$ is infinite iff $g \in \cov^+(\F)$.
	\end{proposition}
	
	\begin{proof}
		Note that $\A^\F_g$ is almost disjoint as for all $f \neq f' \in \F$ we have that $f$ and $f'$ are eventually different.
		Hence, $E^f_g \cap E^{f'}_g$ is finite.
		Next, towards maximality of $A^{\F}_g$, let $A \in \infsubset{\omega}$.
		Since $\F$ is Van Douwen choose $f \in \F$ and $B \in \infsubset{A}$ such that $f \restr B = g \restr B$.
		Thus, $B \subseteq E^f_g$, which shows that $E^f_g \in \A_g^\F$ and $A \cap E^f_g$ is infinite.
		
		Now, assume $g \in \cov(\F)$.
		Choose $\F_0 \in \finsubset{\F}$ as in the definition of $\cov(\F)$ and let $f \in \F \setminus \F_0$.
		But if $E^f_g \in \infsubset{\omega}$ there would be a $f_0 \in \F_0$ such that $E^f_g \cap E^{f_0}_g$ is infinite, i.e.\ $f$ and $f_0$ are not eventually different, a contradiction.
		Thus, $E^f_g \notin \A_g^\F$ and hence $\A_g^\F$ is finite.
		
		Finally, assume $\A_g^\F$ is finite, so choose $\F_0 \in \finsubset{\F}$ with $\A_g^\F = \set{E^f_g}{f \in \F_0 \text{ and } E^f_g \in \infsubset{\omega}}$.
		Since $\A_g^\F$ is maximal we have $\omega \pseudosubseteq \bigcup_{f \in \F_0}E^f_g$.
		Hence, $g \in \cov(\F)$ is witnessed by $\F_0$.
	\end{proof}
	
	\begin{corollary}\label{COR_A_LEQ_AV}
		$\a \leq \av$.
	\end{corollary}
	
	\begin{proof}
		Let $\F$ be a witness for $\av$.
		By the previous proposition it suffices to find a $g \in \cov^+(\F)$.
		Choose a disjoint partition $\omega = \bigcup_{k < \omega} A_k$ into infinite sets and a subset $\set{f_k}{k < \omega}$ from $\F$.
		Then, we define
		$$
			g(n) := f_k(n), \text{ where } n \in A_k.
		$$
		But then $g =^\infty f_k$ for every $k < \omega$, so by Proposition~\ref{PROP_Uncovered_real} we have $g \in \cov^+(\F)$.
	\end{proof}
	
	\begin{remark}
		As in Proposition~\ref{PROP_VanDouwen_Implies_Mad} for any eventually different family $\F$ we may define its trace
		$$
			\trace(\F) := \set{g \in \bairespace}{\A^\F_g \text{ is a (possibly finite) mad family}}.
		$$
		Proposition~\ref{PROP_VanDouwen_Implies_Mad} then implies that $\cov(\F) \subseteq \trace(\F)$ and Van Douwen families satisfy $\trace(\F) = \bairespace$.
		Conversely, $\trace(\F) = \bairespace$ also implies that $\F$ is Van Douwen, for if $A \in \infsubset{\omega}$ and $g: A \to \omega$, let $g^*$ be any extension of $g$ to $\omega$.
		By assumption $\A^\F_{g^*}$ is mad, so choose $f \in \F$ with $E^f_{g^*} \cap A$ infinite.
		But then $f \restr (E^f_{g^*} \cap A) =^\infty g \restr (E^f_{g^*} \cap A)$, i.e.\ $\F$ is Van Douwen.
		Note, that the trace is one of the crucial ingredients of the {\sf ZFC}-construction of a Van Douwen family in \cite{Raghavan_2010}.
	\end{remark}
	
	\section{$\EE_\F(I)$ adds a mad family of size $\max(\left|\F\right|, \left|I\right|)$}
	
	Remember the standard c.c.c.\ forcing $\EE_\F(I)$ for extending an eventually different family $\F$ by $I$-many new eventually different reals:
	\begin{definition}
		Let $\F$ be an e.d.\ family and $I$ an index set.
		Let $\EE_\F(I)$ be the partial order of pairs $(s, E)$, where $s:I \times \omega \partialto \omega$ is a finite partial function and $E \in \finsubset{\F}$.
		For $(s, E) \in \EE_\F(I)$ and $i \in I$ we define the finite partial function $s_i:\omega \partialto \omega$ by $s_i := \set{(n,m)}{(i,n,m) \in s}$ and set $\supp(s) := \set{i \in I}{s_i \neq \emptyset}$.
		For $(s, E), (t, F) \in \EE_\F(I)$ we define $(t, F) \extends (s, E)$ iff
		\begin{enumerate}
			\item $s \subseteq t$ and $E \subseteq F$,
			\item for all $i \neq j \in \supp(s)$ and $n \in \dom(t_i) \setminus \dom(s_i)$ we have $n \notin \dom(t_j)$ or $t_i(n) \neq t_j(n)$,
			\item for all $i \in \supp(s)$, $f \in E$ and $n \in \dom(t_i) \setminus \dom(s_i)$ we have $t_i(n) \neq f(n)$.
		\end{enumerate}
	\end{definition}

	For $\left|I\right| = 1$ Zhang \cite{Zhang_1999} proved that iterating this forcing of uncountable cofinality yields a Van Douwen family.
	Similar density arguments give the same result for the product version:
	
	\begin{lemma} \label{LEM_Force_VanDouwen}
		Let $\F$ be an eventually different family and $I$ an uncountable index set.
		Then
		$$
			\EE_\F(I) \forces \F \cup \dot{\F}_\gen \text{ is a Van Douwen family},
		$$
		where $\dot{\F}_\gen = \set{\dot{f}^i_\gen}{i \in I}$ is the family $I$-many eventually different reals added by $\EE_\F(I)$.
	\end{lemma}
	
	Note, that in contrast to the iteration-version Zhang considered in \cite{Zhang_1999}, with the product-version it is possible to add Van Douwen families of uncountable size with countable cofinality.
	We use Lemma~\ref{LEM_Force_VanDouwen} and a similar argument as in the proof of $\a \leq \av$ (Corollary~\ref{COR_A_LEQ_AV}) to prove that the standard forcing to realize a desired spectrum of $\aE$ also forces $\a$ to have the same spectrum.
	
	\begin{lemma}\label{LEM_MED_Forcing_Adds_Mad_Family}
		Let $\F$ be an eventually different family and $I$ an uncountable index set.
		Then
		$$
		\EE_\F(I) \forces \max(\left|\F\right|, \left|I\right|) \in \spec(\a).
		$$
	\end{lemma}
	
	\begin{proof}
		Choose a disjoint partition $\omega = \bigcup_{i < \omega} A_i$ into infinite sets and a subset $I_0 = \set{i_k}{k < \omega}$ from $I$.
		Let $G$ be $\EE_\F(I)$-generic.
		In $V[G]$ we define
		$$
		g(n) := f^{i_k}_{\gen}(n), \text{ where } n \in A_k.
		$$
		By construction, we have $E_g^{f^{i_k}_\gen} \in \infsubset{\omega}$ for all $k < \omega$.
		We show that also $E_g^f \in \infsubset{\omega}$ for all $f \in \F$.
		So in $V$ fix $f \in \F$ and let $N < \omega$ and $(s, E) \in \EE_\F(I)$.
		Choose $k < \omega$ such that $i_k \notin \supp(s)$ and $n \geq N$ with $n \in A_k$.
		Then $(s \cup \simpleset{(i_k, n, f(n))}, E) \in \EE_\F(I)$, $(s \cup \simpleset{(i_k, n, f(n))}, E) \extends (s,E)$ and
		$$
		(s \cup \simpleset{(i_k, n, f(n))}, E) \forces \dot{g}(n) = \dot{f}^{i_k}_\gen(n) = f(n).
		$$
		Finally, we show that $E_g^{f^{i}_\gen} \in \infsubset{\omega}$ for all $i \in I \setminus I_0$.
		So in $V$ fix $i \in I \setminus I_0$ and let $N < \omega$ and $(s, E) \in \EE_\F(I)$.
		We may assume that $i \in \dom(s)$.
		Choose $k < \omega$ such that $i_k \notin \supp(s)$ and $n \geq N$ with $n \in A_k$ and $n \notin \dom(s_j)$ for all $j  \in \dom(s)$.
		Finally, choose $m \in \omega \setminus \set{f(n)}{f \in E}$.
		Then we have $(s \cup \simpleset{(i_k, n, m), (i, n, m)}, E) \in \EE_\F(I)$, $(s \cup \simpleset{(i_k, n, m), (i, n, m)}, E) \extends (s,E)$ and
		$$
			(s \cup \simpleset{(i_k, n, m), (i, n, m)}, E) \forces \dot{g}(n) = \dot{f}^{i_k}_\gen(n) = m = \dot{f}^{i}_\gen(n).
		$$
		Hence, by Lemma~\ref{LEM_Force_VanDouwen} and Proposition~\ref{PROP_VanDouwen_Implies_Mad} we obtain
		$$
			\EE_\F(I) \forces \A_{\dot{g}}^{\F \cup \dot{\F}_\gen} \text{ is a mad family of size } \max(\left|\F\right|, \left|I\right|),
		$$
		completing the proof.
	\end{proof}

	\section{Spectrum of Van Douwen families}

	In this section, similar to $\a$ \cite{Hechler_1972} and $\aT$ \cite{Brian_2021} we prove that the spectrum of Van Douwen families is closed under singular limits.
	The main idea is that we may glue a sequence of Van Douwen families together in order to obtain a bigger Van Douwen family.
	A similar argument fails for maximal eventually different families as the gluing argument we present here might not preserve maximality.
	Hence, the corresponding Question~\ref{QUE_Closure} for $\aE$ is still open.

	\begin{theorem} \label{THM_Closure}
		$\spec(\av)$ is closed under singular limits.
	\end{theorem}

	\begin{proof}
		Let $\kappa = \cof(\lambda) < \lambda$, $\seq{\lambda_\alpha}{\alpha < \kappa}$ be an increasing sequence of cardinals cofinal in $\lambda$ with $\kappa < \lambda_0$ and $\seq{\F_\alpha}{\alpha < \kappa}$ a sequence of Van Douwen families with $\left|\F_\alpha\right| = \lambda_\alpha$.
		Choose pairwise different elements $\G = \seq{g_\alpha \in \F_0}{\alpha < \kappa}$.
		We construct an eventually different family of functions from $\omega \to \omega \times \omega$ of size $\lambda$ as follows:
		\begin{enumerate}[$\bullet$]
			\item Given $\alpha < \kappa$ and $f \in \F_\alpha$ define $(g_\alpha \times f):\omega \to (\omega \times \omega)$ for $k < \omega$ by
			$$
				(g_\alpha \times f)(k) := (g_\alpha(k), f(k)).
			$$
			\item Given $f_0 \in \F_0 \setminus \G$ and $f_1 \in \F_0$ define $(f_0 \times f_1):\omega \to (\omega \times \omega)$ for $k < \omega$ by 
			$$
				(f_0 \times f_1)(k) := (f_0(k), f_1(k)).
			$$
		\end{enumerate}
		Finally, we define the family $\F$ to be family of all functions from $\omega \to (\omega \times \omega)$ of one of the two forms above.
		Then $\F$ is of size $\lambda$ and we claim that $\F$ is Van Douwen.
		First, we prove that $\F$ is e.d., so we have to consider the following cases:
		\begin{enumerate}[$\bullet$]
			\item Let $\alpha < \kappa$ and $f \neq f' \in \F_\alpha$.
			Since $f$ and $f'$ are e.d.\ choose $K < \omega$ such that $f(k) \neq f'(k)$ for all $k \geq K$.
			But then for every $k \geq K$ we have
			$$
				(g_\alpha \times f)(k) = (g_\alpha(k), f(k)) \neq (g_\alpha(k), f'(k)) = (g_\alpha \times f')(k).
			$$
			\item Let $\alpha \neq \beta < \kappa$ and $f \in \F_\alpha, f' \in \F_\beta$.
			Since $g_\alpha$ and $g_\beta$ are e.d.\ choose $K < \omega$ such that $g_\alpha(k) \neq g_\beta(k)$ for all $k \geq K$.
			But then for every $k \geq K$ we have
			$$
				(g_\alpha \times f)(k) = (g_\alpha(k), f(k)) \neq (g_\beta(k), f'(k)) = (g_\beta \times f')(k).
			$$
			\item Let $\alpha < \kappa$, $f \in \F_\alpha$, $f_0 \in \F_0 \setminus \G$ and $f_1 \in \F_0$.
			Since $g_\alpha$ and $f_0$ are e.d.\ choose $K < \omega$ such that $g_\alpha(k) \neq f_0(k)$ for all $k \geq K$.
			But then for every $k \geq K$ we have
			$$
				(g_\alpha \times f)(k) = (g_\alpha(k), f(k)) \neq (f_0(k), f_1(k)) = (f_0 \times f_1)(k).
			$$
			\item Let $f_0, f_0' \in \F_0 \setminus \G$ and $f_1, f_1' \in \F_0$ with $(f_0', f_1') \neq (f_0, f_1)$.
			W.l.o.g.\ assume $f_0 \neq f_0'$.
			Then $f_0$ and $f_0'$ are e.d., so choose $K < \omega$ such that $f_0(k) \neq f'_0(k)$ for all $k \geq K$.
			But then for every $k \geq K$ we have
			$$
				(f_0 \times f_1)(k) = (f_0(k), f_1(k)) \neq (f'_0(k), f'_1(k)) = (f'_0 \times f'_1)(k).
			$$
		\end{enumerate}
		Hence, it remains to prove that $\F$ is Van Douwen.
		So let $A \in \infsubset{\omega}$ and $h:A \to (\omega \times \omega)$.
		For $i \in 2$ let $p_i(h):A \to \omega$ be the projection of $h$ to the $i$-th component.
		As $\F_0$ is Van Douwen choose $B \in \infsubset{A}$ and $f_0 \in \F_0$ such that $f_0 \restr B = p_0(h) \restr B$.
		We consider the following two cases:
		
		$f_0 \in \G$.
		Choose $\alpha < \kappa$ such that $f_0 = g_\alpha$.
		As $\F_\alpha$ is Van Douwen choose $C \in \infsubset{B}$ and $f \in \F_\alpha$ such that $p_1(h) \restr C = f \restr C$.
		But then for every $k \in C$ we have
		$$
			(g_\alpha \times f)(k) = (g_\alpha(k), f(k)) = (p_0(h)(k), p_1(h)(k)) = h(k).
		$$
		
		Otherwise, $f_0 \in \F_0 \setminus \G$.
		As $\F_0$ is Van Douwen choose $C \in \infsubset{B}$ and $f_1 \in \F_0$ with $p_1(h) \restr C = f_1$.
		But then for every $k \in C$ we have
		$$
			(f_0 \times f_1)(k) = (f_0(k), f_1(k)) = (p_0(h)(k), p_1(h)(k)) = h(k).
		$$
		Hence, in both cases $h$ is infinitely often equal to some element in $\F\restr A$.
	\end{proof}

	\section{Many non Van Douwen families}

	Given a maximal eventually different family $\F$ Raghavan in \cite{Raghavan_2010} introduced the following ideal $\I_0(\F)$ measuring how far $\F$ is from being Van Douwen.

	\begin{definition}\label{DEF_Associated_Ideal}
		Let $\F$ be an eventually different family.
		Then we define
		$$
			\I_0(\F) := \set{A \in \infsubset{\omega}}{\F \restr A \text{ is not a m.e.d. family}} \cup \Fin.
		$$
	\end{definition}

	Note, that $\I_0(\F)$ is proper iff $\F$ is a maximal eventually different family.
	We also verify that it is indeed an ideal:

	\begin{proposition}\label{PROP_Associated_Ideal}
		$\I_0(\F)$ is an ideal.
	\end{proposition}

	\begin{proof}
		Since $\F$ is maximal we have $\omega \notin \I_0(\F)$.
		If $A \in \I_0(\F)$ and $B \in \infsubset{A}$ choose $g:A \to \omega$ eventually different from $\F \restr A$.
		But then $g \restr B: B \to \omega$ is eventually different from $\F \restr B$.
		Hence, $B \in \I_0(\F)$.
		
		Finally, let $A, B \in \I_0(\F)$.
		We may assume that $A \in \infsubset{\omega}$, so choose $g: A \to \omega$ eventually different from $\F \restr A$.
		If $B$ is finite, then any extension of $f:A \to \omega$ to $f^*:(A \cup B) \to \omega$ is eventually different from $\F \restr (A \cup B)$, so assume that $B \in \infsubset{\omega}$ and choose $h:B \to \omega$ eventually different from $\F \restr B$.
		We claim that $g \cup h \restr (A \setminus B):(A \cup B) \to \omega$ is eventually different from $\F \restr (A \cup B)$, so let $f \in \F$.
		By choice of $g$ and $h$ there are $K_0, K_1 < \omega$ such that $g(k) \neq f(k)$ for every $k \in A \setminus K_0$ and $h(k) \neq f(k)$ for every $k \in B \setminus K_1$.
		Hence, $(g \cup h \restr (B \setminus A))(k) \neq f(k)$ for all $k \in (A \cup B) \setminus (K_0 \cup K_1)$.
	\end{proof}

	\begin{corollary}\label{COR_VanDouwen}
		$\F$ is Van Douwen iff $\I_0(\F) = \Fin$.
	\end{corollary}

	We prove that under {\sf CH} any ideal may be realized as the associated $\I_0$-ideal of some maximal eventually different family.
	To achieve this, we need the following two diagonalization lemmata:

	\begin{lemma}\label{LEM_Extend_Witnesses}
		Let $\F = \seq{f_n}{n < \omega}$ be e.d.\ and $A \in \infsubset{\omega}$.
		Then there is $g: A \to \omega$ such that $\F \restr A \cup \simpleset{g}$ is e.d.
	\end{lemma}

	\begin{proof}
		Enumerate $A$ by $\set{a_n}{n < \omega}$.
		Inductively, choose $g(a_n)$ different from $\set{f_m(a_n)}{m < n}$.
		By construction $\F \restr A \cup \simpleset{g}$ is e.d.
	\end{proof}

	\begin{lemma} \label{LEM_Diagonalize_With_Witnesses}
		Let $\I$ be an ideal, $\F = \seq{f_n}{n < \omega}$ be e.d.\ and $\seq{g_n:A_n \to \omega}{n < \omega}$ be such that $A_n \in \I$ and $\F \restr A_n \cup \simpleset{g_n}$ is e.d.\ for all $n < \omega$.
		Further, let $h:B \to \omega$ such that $B \notin \I$ and $\F \restr B \cup \simpleset{h}$ is e.d.
		Then there is $f: \omega \to \omega$ such that
		\begin{enumerate}
			\item $\F \cup \simpleset{f}$ is e.d.,
			\item $(\F \cup \simpleset{f}) \restr A_n \cup \simpleset{g_n}$ is e.d.\ for all $n < \omega$,
			\item $f \restr C = h \restr C$ for some $C \in \infsubset{B}$.
		\end{enumerate}
	\end{lemma}

	\begin{proof}
		We define an increasing sequence of finite partial functions $\seq{s_n}{n < \omega}$ as follows.
		Set $s_0 := \emptyset$.
		Now, let $n < \omega$ and assume $s_n$ is defined.
		By assumption, choose $K < \omega$ such that $\dom(s_n) \subseteq K$ and for all $k \in B \setminus K$ and $m < n$ we have $f_m(k) \neq h(k)$.
		Since $\I$ contains $\Fin$ and $B \notin \I$, the set $B \setminus \bigcup_{m < n}A_m$ is infinite, so choose $k \in B \setminus K$ with $k \notin A_m$ for all $m < n$.
		Now, set $s'_{n + 1} := s_n \cup \simpleset{(k, h(k))}$.
		Finally, if $n \in \dom(s'_{n + 1})$ set $s_{n + 1} := s'_{n + 1}$, otherwise choose $l < \omega$ such that $l \neq f_m(n)$ for all $m < n$ and $l \neq g_m(n)$ for all $m < n$ with $n \in A_m$ and set $s_{n + 1} := s'_{n + 1} \cup \simpleset{(n, l)}$.
		
		Set $f := \bigcup_{n < \omega} s_n$.
		Then, $f :\omega \to \omega$ as $n \in \dom(s_{n + 1})$ for all $n < \omega$.
		Further, we verify (1-3):
		\begin{enumerate}
			\item Let $m < \omega$, then for every $n > m$ and $k \in \dom(s_{n + 1}) \setminus \dom(s_n)$ we compute that $f(k) = s_{n + 1}(k) \neq f_m(k)$ by choice of $K$ or $l$, i.e.\ $f$ and $f_m$ are e.d.,
			\item Let $m < \omega$, then for every $n > m$ and $k \in (A_m \cap \dom(s_{n + 1})) \setminus \dom(s_n)$ we compute that $f(k) = s_{n + 1}(k) \neq g_m(k)$ by choice of $k$ or $l$, i.e.\ $f \restr A_m$ and $g_m$ are e.d.,
			\item For every $n < \omega$ there is a $k \in \dom(s_{n + 1}) \setminus \dom(s_n)$ such that $f(k) = s_{n + 1}(k) = h(k)$, i.e. $f \restr C = h \restr C$ for some $C \in \infsubset{B}$.
		\end{enumerate}
		Hence, $f$ is as desired.
	\end{proof}

	\begin{theorem}\label{THM_Realize_Ideal}
		Assume {\sf CH} and let $\I$ be an ideal.
		Then there is a maximal eventually different family $\F$ such that $\I = \I_0(\F)$.
	\end{theorem}

	\begin{proof}
		Enumerate all functions $\set{h_\alpha:B_\alpha \to \omega}{\alpha < \aleph_1}$, where $B_\alpha \notin \I$, and $\I$ by $\seq{A_\alpha}{\alpha < \aleph_1}$.
		We inductively construct a m.e.d.\ family $\seq{f_\alpha}{\alpha < \aleph_1}$ and functions $\seq{g_\alpha:A_\alpha \to \omega}{\alpha < \aleph_1}$ while preserving the following properties for every $\alpha < \aleph_1$:
		\begin{enumerate}
			\item $\F_{<\alpha} := \set{f_\beta}{\beta < \alpha}$ is e.d.,
			\item $\F_{<\alpha} \restr A_\beta \cup \simpleset{g_\beta}$ is e.d. for all $\beta < \alpha$,
			\item If $\F_{<\alpha} \restr B_\alpha \cup \simpleset{h_\alpha}$ is e.d., then $f_\alpha \restr C = h_\alpha \restr C$ for some $C \in \infsubset{B_\alpha}$.
		\end{enumerate}
		Let $\alpha < \aleph_1$.
		By induction assumption we may apply Lemma~\ref{LEM_Diagonalize_With_Witnesses} to $\F_{<\alpha}$, $\seq{g_\beta:A_\alpha \to \omega}{\beta < \alpha}$ and $h_\alpha: B_\alpha \to \omega$ to obtain $f_\alpha:\omega \to \omega$ such that
		\begin{enumerate}
			\item $\F_{<\alpha} \cup \simpleset{f_\alpha}$ is e.d.,
			\item $(\F_{<\alpha} \cup \simpleset{f_\alpha}) \restr A_\beta \cup \simpleset{g_\beta}$ is e.d.\ for all $\beta < \alpha$,
			\item If $\F_{<\alpha} \restr B_\alpha \cup \simpleset{h_\alpha}$ is e.d., then $f_\alpha \restr C = h_\alpha \restr C$ for some $C \in \infsubset{B_\alpha}$.
		\end{enumerate}
		Finally, by Lemma~\ref{LEM_Extend_Witnesses} choose $g_\alpha:A_\alpha \to \omega$ such that $(\F_{<\alpha} \cup \simpleset{f_\alpha}) \restr A_\alpha \cup \simpleset{g_\alpha}$ is e.d.
		
		Then, by (1) $\F := \seq{f_\alpha}{\alpha < \aleph_1}$ is e.d.\ and we claim that $\I = \I_0(\F)$.
		First, let $A \in \I$.
		Choose $\alpha < \aleph_1$ such that $A = A_\alpha$.
		By (2) we have that $\F \restr A \cup \simpleset{g_\alpha}$ is e.d., which witnesses $A \in \I_0(\F)$.
		
		Secondly, let $B \notin \I$ and assume $B \in \I_0(\F)$.
		Choose $h:B \to \omega$ such that $\F \restr B \cup \simpleset{h}$ is e.d.
		Choose $\alpha < \aleph_1$ such that $B = B_\alpha$ and $h = h_\alpha$.
		Then also $\F_{<\alpha} \restr B \cup \simpleset{h}$ is e.d.
		Thus, by (3) we have $f_\alpha \restr C = h \restr C$ for some $C \in \infsubset{B}$, which contradicts that $f_\alpha \restr B$ and $h$ are e.d.
	\end{proof}

	Notice, that with Theorem~\ref{THM_Realize_Ideal} under $\sf CH$ we may construct Van Douwen families $\F$, so that $\I_0(\F)$ is a maximal ideal.
	These kind of maximal eventually families are in some sense as far as possible away from being Van Douwen while still being maximal.
	Hence, we may define
	
	\begin{definition}
		A maximal eventually family is called very non Van Douwen iff $\I_0(\F)$ is a maximal ideal.
		Equivalently, for any $A \in \infsubset{\omega} \setminus \cofin(\omega)$ exactly one of the following holds:
		\begin{enumerate}
			\item Either there is $g:A \to \omega$ such that $\F \restr A \cup \simpleset{g}$ is e.d.,
			\item or there is $g: A^c \to \omega$ such that $\F \restr A^c \cup \simpleset{g}$ is e.d.
		\end{enumerate}
	\end{definition}

	\noindent Similar to Van Douwen's question we may ask:
	
	\begin{question}
		Does there always exist a very non Van Douwen m.e.d.\ family?
	\end{question}

		Further, under $\sf CH$ these very non Van Douwen m.e.d.\ families may have a P-ideal as their associated ideal, whose maximality is preserved by Sacks \cite{Shelah_2017} or Miller \cite{Miller_1984} forcing.
	Note, this does not imply that the maximality of $\F$ is preserved, as the associated ideal interpreted in the generic extension may grow compared to the associated ideal interpreted in the ground model.
	
	\begin{definition}
		Let $\F$ be a maximal eventually different family and $\PP$ be a forcing.
		We say that $\PP$ preserves $\I_0(\F)$ or $\I_0(\F)$ is $\PP$-indestructible iff for every $\PP$-generic $G$ in $V[G]$ we have that $\I_0(\F) = \langle \I_0(\F)^V \rangle$, where $\langle\cdot\rangle$ is the generated ideal.
	\end{definition}

	\begin{question}
		For which forcings may we construct non Douwen m.e.d.\ families, so that their associated ideals are preserved?
	\end{question}
	
	These considerations are particularly interesting towards a possible proof of the consistency of $\aE < \av$, as in this case a non Van Douwen m.e.d.\ family which remains maximal under some forcing which destroys the Van Douwenness of some other family is desirable.
	
	As a proof of concept for the remainder of this paper we show that assuming {\sf CH} we may indeed construct maximal eventually different families with associated ideal indestructible by (any) Sacks forcing.
	In order to construct such a family, we use the framework developed Fischer and the author in \cite{FischerSchembecker_2023}.
	That is, to obtain a maximal eventually different family with its associated ideal indestructible by any kind of Sacks-forcing it suffices to construct such a family whose associated ideal is indestructible by $\SS^{\aleph_0}$, the fully supported product of Sacks-forcing of size $\aleph_0$:
	
	\begin{corollary}\cite{FischerSchembecker_2023}\label{COR_Indestructibility_Implication}
		Let $\F$ be a maximal eventually different family such that its associated ideal $\I_0(\F)$ is $\SS^{\aleph_0}$-indestructible.
		Then $\I_0(\F)$ is also indestructible by any countable supported product or iteration of Sacks-forcing.
	\end{corollary}
	
	Thus, it suffices to restrict our attention to $\SS^{\aleph_0}$.
	We will need to refine Lemma~\ref{LEM_Diagonalize_With_Witnesses} to also diagonalize against $\SS^{\aleph_0}$-names $\dot{h}:\dot{B} \to \omega$, which may extend the associated ideal in the Sacks-extension.
	$\SS^{\aleph_0}$ satisfies Axiom A/has fusion, we assume the reader is familiar with the usual fusion order of $\SS^{\aleph_0}$ and the notion of suitable functions.
	For more details see \cite{FischerSchembecker_2023} or \cite{Kanamori_1980}.
	
	\begin{lemma} \label{LEM_Diagonalize_With_Sacks_Witnesses}
		Let $\I$ be an ideal, $\F = \seq{f_n}{n < \omega}$ be e.d.\ and $\seq{g_n:A_n \to \omega}{n < \omega}$ be such that $A_n \in \I$ and $\F \restr A_n \cup \simpleset{g_n}$ is e.d.\ for all $n < \omega$.
		Further, let $p \in \SS^{\aleph_0}$ and $\dot{h}$ be a $\SS^{\aleph_0}$-name such that
		\[
			p \forces \dot{h}:\dot{B} \to \omega \text{ such that } \dot{B} \notin \langle\I\rangle \text{ and } \F \restr \dot{B} \cup \simpleset{\dot{h}} \text{ is e.d.}
		\]
		Then there is $q \extends p$ and $f: \omega \to \omega$ such that
		\begin{enumerate}
			\item $\F \cup \simpleset{f}$ is e.d.,
			\item $(\F \cup \simpleset{f}) \restr A_n \cup \simpleset{g_n}$ is e.d.\ for all $n < \omega$,
			\item $q \forces f \restr \dot{C} = \dot{h} \restr \dot{C}$ for some $\dot{C} \in \infsubset{\dot{B}}$.
		\end{enumerate}
	\end{lemma}

	\begin{proof}
		As before we define an increasing sequence of finite partial functions $\seq{s_n}{n < \omega}$ but now alongside a fusion sequence $\seq{p_n}{n < \omega}$ below $p$.
		Set $s_0 := \emptyset$ and $p_0 := p$.
		Now, let $n < \omega$ and assume that $s_n$ and $p_n$ have been defined.
		As in the proof of Lemma~\ref{LEM_Diagonalize_With_Witnesses} we may assume that $n \in \dom(s_n)$.
		Let $\seq{\sigma_m}{m < M}$ enumerate all suitable functions for $p_n$ and $n$.
		We define an increasing sequence of finite partial functions $\seq{t_m}{m \leq M}$ extending $s_n$ and $\leq_n$-extensions $\seq{q_m}{m \leq M}$ of $p_n$ as follows:
		
		Set $t_0 := s_n$ and $q_0 := p_n$.
		Now, let $m < M$ and assume $s_m$ and $q_m$ have been defined.
		By assumption we have
		\[
			q_m \restr \sigma_m \forces \dot{h}:\dot{B} \to \omega \text{ such that } \dot{B} \notin \langle\I\rangle \text{ and } \F \restr \dot{B} \cup \simpleset{\dot{h}} \text{ is e.d.}
		\]
		As in the proof of Lemma~\ref{LEM_Diagonalize_With_Witnesses} choose $r_m \extends q_m \restr \sigma_m$ and $k, l < \omega$ with $k \notin \dom(t_n) \cup 
		\bigcup_{\bar{n} < n} A_{\bar{n}}$ and
		\[
			r_m \forces k \in \dot{B} \text{ and } \dot{h}(k) = l \neq f_{\bar{n}}(k) \text{ for all } \bar{n} < n.
		\]
		Finally, we set $t_{m + 1} := t_m \cup \simpleset{(k, l)}$ and define $q_{m+1} \leq_n q_m$ with $q_{m + 1} \restr \sigma_m = r_m$ by 
		\[
			q_{m+1}(\alpha) :=
			\begin{cases}
				r_m(\alpha) \cup \bigcup\set{q_m(\alpha) \restr \sigma_{\bar{m}}}{\bar{m} < M, \bar{m} \neq m} &\text{if } \alpha < n,\\
				r_m(\alpha) & \text{otherwise}.
			\end{cases}
		\]
		Now, set $p_{n+1} := q_M$ and $s_{n+1} := t_M$.
		By construction we have that
		\[
			p_{n+1} \forces \exists k \in \dom(s_{n + 1}) \setminus \dom(s_n) \text{ such that } \dot{h}(k) = s_{n+1}(k)
		\]
		and for every $k \in \dom(s_{n + 1}) \setminus \dom(s_n)$ and $\bar{n} < n$ we have $s_{n+1}(k) \neq f_{\bar{n}}(k)$.
		
		Set $f := \bigcup_{n < \omega} s_n$ and let $q$ be the fusion of $\seq{p_n}{n < \omega}$.
		As before, we have that $f:\omega \to \omega$ and $f$ and $q$ satisfy (1-3).
	\end{proof}

	\begin{theorem}\label{THM_Realize_Ideal_Sacks}
		Assume {\sf CH} and let $\I$ be an ideal.
		Then there is a maximal eventually different family such that $\I = \I_0(\F)$ and $\I_0(\F)$ is indestructible by any countable support iteration or product of Sacks-forcing.
	\end{theorem}

	\begin{proof}
		The proof is essentially the same as the proof of Theorem~\ref{THM_Realize_Ideal}.
		The only difference is instead of enumerating $\seq{h_\alpha:B_\alpha\to\omega}{\alpha < \aleph_1}$, where $B_\alpha \notin \I$, we enumerate all pairs $\seq{(p_\alpha, \dot{h}_\alpha)}{\alpha < \aleph_1}$, where $p \in \SS^{\aleph_0}$ and $\dot{h}_\alpha$ is a $\SS^{\aleph_0}$-name such that
		\[
			p_\alpha \forces \dot{h}_\alpha:\dot{B} \to \omega \text{ such that } \dot{B} \notin \langle\I\rangle.
		\]
		We cannot enumerate all such $\SS^{\aleph_0}$-names in order type $\aleph_1$, but using continuous reading of names and {\sf CH} already $\aleph_1$-many nice names suffice, see \cite{FischerSchrittesser_2021} for more details.
		Now, argue the same way as in the proof of Theorem~\ref{THM_Realize_Ideal} using Lemma~\ref{LEM_Diagonalize_With_Sacks_Witnesses} instead to iteratively construct a family $\F$ with $\I = \I_0(\F)$ and $\I_0(\F)$ is indestructible by $\SS^{\aleph_0}$.
		By Corollary~\ref{COR_Indestructibility_Implication} $\I_0(\F)$ is also indestructible by any countably supported product or iteration of Sacks-forcing.
	\end{proof}

	As a corollary we show that in the iterated Sacks-model over a model of {\sf CH} every $\aleph_1$-generated ideal, i.e.\ with a basis of size $\aleph_1$, is realized as the associated ideal of some maximal eventually different family.

	\begin{corollary}\label{COR_Realize_Ideal_Sacks_Model}
		In the iterated Sacks-model for every $\aleph_1$-generated ideal $\I$ there is a maximal eventually different family $\F$ such that $\I = \I_0(\F)$.
	\end{corollary}

	\begin{proof}
		Let $G$ be $\SS_{\omega_2}$-generic, where $\SS_{\omega_2}$ is the countable support iteration of Sacks-forcing of length $\omega_2$ over a model of {\sf CH}.
		Every $\aleph_1$-generated generated ideal $\I$ has a basis $\B$ of size $\aleph_1$ added at an intermediate step $V[G_\alpha]$ with $\alpha < \omega_2$ of the iteration.
		In $V[G_\alpha]$ {\sf CH} still holds, so by the previous theorem there is a maximal eventually different family $\F$ in $V[G_\alpha]$ with $\langle\B\rangle = \I_0(\F)$ and $\I_0(\F)$ is indestructible by any countably supported iteration or product of Sacks-forcing.
		But the quotient forcing $\SS_{\omega_2} / \SS_\alpha$ is isomorphic to $\SS_{\omega_2}$, hence by the indestructibility of $\I_0(\F)$ in $V[G]$ we have that $\I_0(\F) = \langle\langle\B\rangle^{V[G_\alpha]}\rangle = \langle \B \rangle = \I$.
	\end{proof}
	
	\begin{corollary}\label{COR_Very_Non_Van_Douwen}
		In the iterated Sacks-model there is a very non Van Douwen m.e.d.\ family of size $\aleph_1$.
		In particular their existence is consistent with $\c = \aleph_2$.
	\end{corollary}

	\begin{proof}
		In $V$ by {\sf CH} there is a P-ideal (i.e.\ the dual ideal of a P-point), which are preserved by iterations of Sacks-forcing \cite{Shelah_2017}.
		Thus, by Theorem~\ref{THM_Realize_Ideal_Sacks} there is a maximal eventually different family $\F$ such that in the iterated Sacks-model we have $\I_0(\F) = \langle \I \rangle$, where $\langle\I\rangle$ is a maximal ideal.
		Hence, $\F$ is very non Van Douwen also in the generic extension.		
	\end{proof}
	
	Finally, we note that we may also obtain very non Douwen m.e.d.\ families of size $\aleph_1$ with arbitrarily large continuum in the product Sacks-models.
	The argument is the same as above together with the fact that product Sacks-forcing also preserves P-ideals \cite{Laver1984}.
	
	\bibliographystyle{plain}
	\bibliography{refs}
	
\end{document}